\documentclass[letterpaper, 10 pt, conference]{ieeeconf}  

\IEEEoverridecommandlockouts                              

\usepackage{graphics} 
\usepackage{epsfig} 
\usepackage{mathtools}
\usepackage{times} 
\usepackage{amsmath} 
\usepackage{amssymb}  
\usepackage{tabularx}
\usepackage{cite,balance}

\newtheorem{theorem}{Theorem}
\newtheorem{definition}{Definition}
\newtheorem{proposition}{Proposition}
\newtheorem{remark}{Remark}
\newtheorem{example}{Example}
\usepackage{comment}
\usepackage{optidef} 
\graphicspath{{Images/}}
\usepackage{caption}
\usepackage{subcaption}
\usepackage{booktabs}
\usepackage{comment}
\usepackage[]{algorithm2e}
\RestyleAlgo{ruled} 

\usepackage[font=small,labelfont=bf]{caption}
\usepackage[colorlinks = true,
            linkcolor = blue,
            urlcolor  = blue,
            citecolor = blue,
            anchorcolor = blue]{hyperref}

\newcommand{\tr}{{{\mathsf T}}}

\title{\LARGE \bf
Iterative Inner/outer Approximations for Scalable Semidefinite Programs using Block Factor-width-two Matrices
}

\author{Feng-Yi Liao and Yang Zheng$^\dagger$
\thanks{}
\thanks{
        $^\dagger$F. Liao and Y. Zheng are with ECE Department, University of California, San Diego. {Emails: fliao@ucsd.edu; zhengy@eng.ucsd.edu}}%
}

\allowdisplaybreaks
\usepackage{setspace}
\begin{document}

\maketitle
\thispagestyle{empty}
\pagestyle{empty}

\begin{abstract}
In this paper, we propose iterative inner/outer approximations based on a recent notion of block factor-width-two matrices for solving semidefinite programs (SDPs). 
Our inner/outer approximating algorithms generate a sequence of upper/lower bounds of increasing accuracy for the optimal SDP cost. The block partition in our algorithms offers flexibility in terms of both numerical efficiency and solution quality, which includes the approach of scaled diagonally dominance (SDD) approximation as a special case. We discuss both the theoretical results and numerical implementation in detail. Our main theorems guarantee that the proposed iterative algorithms generate monotonically decreasing upper (increasing lower) bounds. 
Extensive numerical results confirm our findings.

\end{abstract}

\section{Introduction}
Semidefinite programs (SDPs) are a class of convex optimization problems over the positive semidefinite (PSD) cone. The standard primal and dual SDPs are in the form of 
\begin{subequations} \label{eq:SDPs}
    \begin{align}
        p^{\star} := \min_{X} \quad & \langle C, X\rangle \nonumber\\
        \mathrm{subject~to} \quad & \langle A_i, X\rangle = b_i, \quad i = 1, \ldots, m,   \label{eq:SDP-primal}\\
        & X \in \mathbb{S}^n_+, \nonumber
    \end{align}
     \begin{align}
        d^{\star} := \max_{y, Z} \quad & b^\tr y \nonumber \\
        \mathrm{subject~to} \quad & Z + \sum_{i=1}^m A_i y_i = C, \label{eq:SDP-dual} \\
        & Z \in \mathbb{S}^n_+, \nonumber
    \end{align}
\end{subequations}
where $b \in \mathbb{R}^m, C, A_1, \ldots, A_m \in \mathbb{S}^n$ are the problem data, $\mathbb{S}^n_+$ denotes the set of $n \times n$ PSD matrices (we also write $X \succeq 0$ to denote $X \in \mathbb{S}^n_+$ when the dimension is clear from the context), and $\langle \cdot,\cdot \rangle$ denotes the standard inner product in an approximate space. We assume the strong duality holds for the primal and dual SDPs \eqref{eq:SDPs}, i.e., $p^{\star} = d^{\star}$. 

Semidefinite optimization~\eqref{eq:SDP-primal} and \eqref{eq:SDP-dual}  is a powerful computational tool in control theory \cite{boyd1994linear}, combinatorial problems \cite{sotirov2012sdp}, non-convex polynomial optimization \cite{blekherman2012semidefinite}, and many other areas \cite{vandenberghe1996semidefinite}. While interior-point methods can solve SDPs in polynomial time to arbitrary accuracy in theory \cite{vandenberghe1996semidefinite}, they are not scalable to address many large-scale problems of practical interest \cite{zheng2021chordal,majumdar2020recent}. One main difficulty is due to the need of storage, computation, and factorization of a large matrix at each iteration of interior-point methods. Existing general-purpose SDP solvers (including the state-of-the-art solver MOSEK~\cite{aps2019mosek}) are limited to medium-scale problem instances (with $n$ less than 1000 and $m$ being a few hundreds in \eqref{eq:SDPs}). 

Overcoming the challenge of scalability has received much attention \cite{zheng2021chordal,majumdar2020recent}. One class of approaches is to develop efficient algorithms based on first-order methods. For instance, a general conic solver based on alternating direction method of multipliers (ADMM) was developed in~\cite{o2016conic}, and a sparse conic solver based on ADMM for SDPs with chordal sparsity was developed in \cite{zheng2020chordal,zheng2016cdcs}; see \cite[Section 3]{zheng2021chordal}~for~a recent overview. 
While first-order methods considerably speed up the computational time at each iteration,  achieving solutions of high accuracy  remains a central challenge and may require unacceptable many iterations. Therefore, first-order methods are mainly suitable for applications that only require solutions of moderate accuracy. 

Another class of approaches for efficiency improvement is to (equivalently or approximately) decompose a large PSD matrix into the sum of smaller PSD matrices that are easier to handle~\cite{zheng2021chordal}.
%
%
Specifically, one can try to decompose $X=\sum_{i=1}^{t} Q_i$, where $Q_i \succeq 0$ are nonzero
only on a certain (and, ideally, small) principal submatrix. When $X$ has~a special chordal sparsity pattern, such a decomposition is equivalent~\cite{zheng2021chordal,vandenberghe2015chordal,zheng2021sum}. In general, the decomposition above gives an inner approximation of the PSD cone $\mathbb{S}^n_+$. One widely used strategy is so-called \textit{scaled-diagonally dominant (SDD) matrices}~\cite{{boman2005factor}}, where each $Q_i$ only involves a $2 \times 2$ nonzero principal matrix that is equivalent to a second-order cone constraint. Second-order cone programs (SOCPs) admit much more efficient algorithms than SDPs. This scalability feature is one main motivation in the recent studies \cite{ahmadi2019dsos,bertsimas2020polyhedral,ahmadi2017optimization,wang2021polyhedral,ahmadi2017sum,roig2022globally}. In particular, this idea has been extensively used in the context of sum-of-squares optimization \cite{ahmadi2019dsos}. 
%
%
While the SDD approximation brings considerable computational efficiency in solving~\eqref{eq:SDPs}, the solution might be very conservative \cite{ahmadi2019dsos,zheng2021chordal}. Several iterative methods have been further proposed to improve solution quality, such as adding linear cuts or second-order cuts \cite{ahmadi2017optimization,wang2021polyhedral,bertsimas2020polyhedral}, and 
%
basis pursuit searching \cite{ahmadi2017sum,roig2022globally}. 
These methods \cite{ahmadi2017optimization,bertsimas2020polyhedral,wang2021polyhedral,ahmadi2017sum,roig2022globally} solve a linear program (LP) or a SOCP at each iteration, but may require many iterations to get a reasonable good solution  (if possible).

Recently, a new block extension of SDD matrices, called \textit{block factor-width-two} matrices, has been introduced in \cite{zheng2022block,sootla2019block}, where $Q_i \succeq 0$ involves a $2 \times 2$ block principle matrix. This notion works on block-partitioned matrices, and the block partition brings flexibility in terms of both solution quality and numerical efficiency in solving~\eqref{eq:SDPs}, as demonstrated extensively in \cite{zheng2022block,sootla2019block}. 
In this paper, we further develop iterative inner/outer approximations based on the new notion of \textit{block factor-width-two} matrices. Our iterative algorithms generalize the results in \cite{ahmadi2017sum} to the case of \textit{block factor-width-two} matrices and include \cite{ahmadi2017sum} as a special case (cf. Algorithms \ref{Algorithm:inner approximation}-\ref{Algorithm:outer approximation}). Our algorithms provide a sequence of upper and lower bounds of increasing accuracy on the optimal SDP cost $p^{\star}=d^{\star}$ (cf. Propositions \ref{prop:decreasing}-\ref{prop:increasing} and Theorems \ref{theorem:strictly decreasing}-\ref{theorem:strictly increasing}). Numerical results on independent stable set, binary quadratic optimization, and random SDPs confirm the performance of our iterative inner/outer approximations. 



The rest of this paper is organized as follows. In Section \ref{section:background}, we review the SDD and block SDD matrices. The iterative inner/outer approximations and their solution quality are presented in Section \ref{section: inner-outer-approximation} and Section \ref{sec:Outer}. 
In Section \ref{section:Numerical Results}, we present the numerical results. 
Section \ref{section:Conclution} concludes the paper. 


\section{Preliminaries and Problem Statement} \label{section:background}

In this section, we first give a brief overview of the existing approximation strategies for the PSD cone $\mathbb{S}^n_+$, including (scaled-) diagonally dominant matrices \cite{boman2005factor} and their block extensions \cite{zheng2022block}. These approximation strategies have shown promising computational efficiency improvement to solve \eqref{eq:SDP-primal}-\eqref{eq:SDP-dual}, but may also suffer from conservatism in solution quality \cite{zheng2022block,ahmadi2019dsos}. We then present the problem statement of improving the approximation quality via iterative algorithms.    

\subsection{DD and SDD matrices}

The class of (scaled-) diagonally dominant matrices is defined as follows \cite{boman2005factor}.  
\begin{definition}
    A symmetric matrix $A = [a_{ij}] \in \mathbb{S}^{n}$ is diagonally dominant (DD) if and only if
    \vspace{-1mm}
    \begin{equation} \label{eq:dd-matrices}
        a_{ii} \geq \sum_{j \neq i} |a_{ij}|,\  i= 1, 2,\ldots, n.
        \vspace{1mm}
    \end{equation}
\end{definition}

\begin{definition}
    A symmetric matrix $A$ is scaled diagonally dominant (SDD) if and only if there exists a diagonal matrix $D$ with positive entries such that $DAD$ is DD.
\end{definition}

We denote the set of of $n \times n$ DD matrices as $DD_n$ and the set of $n \times n$ SDD matrices  as $SDD_n$. It is known that the following inclusion holds (e.g., by Gershgorin's circle theorem) \cite{boman2005factor}
$$
    DD_n \subseteq SDD_n \subseteq \mathbb{S}^n_+.
$$

An SDD matrix $A$ has an equivalent characterization as a factor-width-two matrix, i.e., $A = VV^\tr$ where each column of $V$ contains at most two non-zero elements \cite{boman2005factor}. Furthermore, 
\begin{equation} \label{eq:sdd-socp}
    A \!\in\! SDD_n \Leftrightarrow A \!=\!\!\! \sum_{1 \leq i<j \leq n} E_{ij}^\tr M_{ij} E_{ij}, \text{with} \; M_{ij}\! \in\! \mathbb{S}^2_+,
\end{equation}
where $E_{ij} \in \mathbb{R}^{2 \times n}$ with $i$th entry in row 1 and $j$th entry in row 2 being 1 and other entries being zero. 

It is not difficult to see that \eqref{eq:dd-matrices} can be written as a set of linear constraints and \eqref{eq:sdd-socp} can be reformulated to a set of second-order constraints. Thus approximating $\mathbb{S}^n_+$ by $DD_n$ ($SDD_n$ respectively) in \eqref{eq:SDPs} becomes a linear program (second-order cone program, respectively), for which very efficient algorithms exist~\cite{alizadeh2003second}. This computational feature is one main motivation in the recent studies \cite{ahmadi2017sum,ahmadi2019dsos}.



\subsection{Block SDD matrices} \label{sec:BSDD}

The characterization in \eqref{eq:sdd-socp} only involves $2 \times 2$ PSD~matrices. A recent study has introduced a block extension to bridge the gap between $SDD_n$ and $\mathbb{S}^n_+$ \cite{zheng2022block}. The main idea is to allow \eqref{eq:sdd-socp} use $2 \times 2$ \textit{block} matrices. To introduce this block extension, we need to define \textit{block-partitioned matrices}.

Given a set of integers $\alpha=\{\alpha_1,\alpha_2,\ldots,\alpha_p\}$ with $\sum_{i=1}^p \alpha_i = n$, we say a matrix $A \in \mathbb{R}^{n \times n}$ is block-partitioned by $\alpha$ if we can write $A$  as 
\begin{equation}
    \begin{bmatrix}A_{11} & A_{12} & \ldots &A_{1p} \\ A_{21} & A_{22} & \ldots & A_{2p} \\ \vdots & \vdots &  \ddots & \vdots \\ A_{p1} & A_{p2} & \ldots & A_{pp} \end{bmatrix},
\end{equation}
where $A_{ij} \in \mathbb{R}^{\alpha_i\times \alpha_j}, \forall i,j = 1,2,\ldots,p$.  
%
Given a partition $\alpha =\{\alpha_{1}, \alpha_{2}, \cdots, \alpha_{p}\}$, we define a 0/1 index matrix $E_{i}^{\alpha}$ as 
\begin{equation}
    E_{i}^{\alpha} = \begin{bmatrix} 0 & 0 & \ldots & I_{\alpha_i} & \ldots & 0 \end{bmatrix} \in \mathbb{R}^{\alpha_i \times n}, 
\end{equation}
and another matrix 
$$
    E_{ij}^{\alpha} = \begin{bmatrix}
    E_{i}^{\alpha} \\E_{j}^{\alpha}
    \end{bmatrix} \in  \mathbb{R}^{(\alpha_i+\alpha_j) \times n}, \quad i \neq j. 
$$
It is clear that $E_{ij}$ in \eqref{eq:sdd-socp} is the same as $E_{ij}^{\alpha}$ when $\alpha = \{1,1, \ldots, 1\}$ (i.e., the trivial partition). Now, we are ready to introduce the notion of block factor-width-two matrices. 

\begin{definition}[{\!\cite{zheng2022block}}] \label{def:alpha_block_factor_width2}
    A symmetric matrix $A$ with partition $\alpha =\{\alpha_{1}, \alpha_{2}, \cdots, \alpha_{p}\}$ belongs to \textit{block factor-width-two} matrices, denoted as $\mathcal{FW}^n_{\alpha,2}$, if there exist  $X_{ij}$ such that
    \begin{equation} \label{eq:factor-width-two}
        A = \sum_{1 \leq i<j \leq p}^{p}
        (E_{ij}^{\alpha})^\tr X_{ij}E_{ij}^{\alpha}, \quad \text{with} \; \; X_{ij} \in \mathbb{S}^{\alpha_i + \alpha_j}_{+}.
    \end{equation}
    \vspace{1mm}
\end{definition}

We note that a matrix can be partitioned in different ways. This flexibility in block factor-width-two matrices can be used to build a converging hierarchy of approximations for $\mathbb{S}^n_+$~\cite[Theorem 2]{zheng2022block}.  For example, three possible partitions for a $10 \times 10$ matrix are  $ \alpha=\{1,1,\ldots,1\}$, $\beta=\{2,2,2,2,2\}$, $\gamma=\{4,4,2\}$, for which we have that  \cite{zheng2022block} 
$$
\label{block_hierarchy}SDD_{10} = \mathcal{FW}^{10}_{\alpha,2} \subseteq \mathcal{FW}^{10}_{\beta,2} \subseteq \mathcal{FW}^{10}_{\gamma,2}.
$$ 
This inclusion relation is illustrated in Figure \ref{fig: feasible-region-different-partitions}, which~shows the feasible set of $x$ and $y$ for which the $10 \times 10$ symmetric matrix $I_{10}+xA+yB$ ($A$ and $B$ are two random generated $10 \times 10$ symmetric matrices) belongs to PSD,  $\mathcal{FW}^{10}_{\alpha,2}$, $\mathcal{FW}^{10}_{\beta,2}$, and $\mathcal{FW}^{10}_{\gamma,2}$. 

In particular, we say a partition $\alpha$ is a \textit{finer} partition of $\beta$, denoted as $ \alpha \sqsubseteq \beta$, if $\alpha$ can be formed by breaking some blocks in $\beta$ (or equivalently, $\beta$ can be formed by merging some blocks in $\alpha$); see \cite[Definition 1]{zheng2022block}. We have the following theorem.
\begin{theorem}[{\!\!\cite[Theorem 2]{zheng2022block}}] \label{theo:block-factor-width-two}
Given $\{1,1,\ldots,1\} \sqsubseteq \alpha \sqsubseteq \beta \sqsubseteq \gamma = \{\gamma_1,\gamma_2\} $ with $\gamma_1 + \gamma_2 = n$, we have a converging hierarchy of inner and outer approximations
    \begin{equation}
    \begin{aligned}
 &SDD_n \subseteq \mathcal{FW}^{n}_{\alpha,2} \subseteq \mathcal{FW}^{n}_{\beta,2}  \subseteq \mathcal{FW}^{n}_{\gamma,2} = \mathbb{S}^{n}_{+} \\
    &= (\mathcal{FW}^{n}_{\gamma,2})^{*} \!\subseteq \!(\mathcal{FW}^{n}_{\beta,2})^{*} \!\subseteq \!(\mathcal{FW}^{n}_{\alpha,2})^{*} \subseteq (SDD_n)^*,
    \end{aligned}
\end{equation}
where $(\mathcal{FW}^{n}_{\alpha,2})^{*}$ denotes the dual of $\mathcal{FW}^{n}_{\alpha,2}$. 
\end{theorem}

\begin{figure} 
    \centering
    \includegraphics[width=0.3\textwidth]{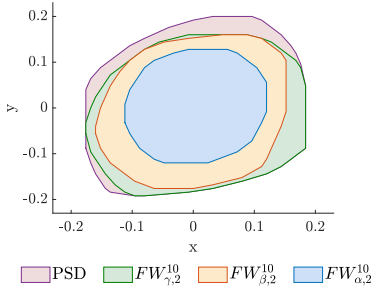}
        \caption{Feasible region of the set of $x$ and $y$ for which the $10 \times 10$ $I_{10}+xA+yB$ belongs to PSD,  $\mathcal{FW}^{10}_{\alpha,2}$, $\mathcal{FW}^{10}_{\beta,2}$, and $\mathcal{FW}^{10}_{\gamma,2}$, where $\alpha=\{1,1,\ldots,1\}$, $\beta=\{2,2,2,2,2\}$, $\gamma=\{4,4,2\}$.}
    \label{fig: feasible-region-different-partitions}
\end{figure}

\subsection{Problem statement}

In \cite{ahmadi2017sum}, we have seen significant numerical efficiency improvements by approximating $\mathbb{S}^n_+$ using $DD_n$ and $SDD_n$ for solving the SDP \eqref{eq:SDPs}, but the solution quality can be unsatisfactory. As shown in Theorem \ref{theo:block-factor-width-two}, the block factor-width-two matrices can improve the solution quality by using a coarser partition $\beta$ \cite{zheng2022block}. This leads to larger PSD constraints shown in \eqref{eq:factor-width-two}, potentially compromising the numerical efficiency.    

In this work, we aim to develop iterative inner and outer approximations for solving the SDP \eqref{eq:SDPs} and at each iteration the partition $\alpha$ is fixed. In this way, we solve the SDP \eqref{eq:SDPs} by  solving  smaller SDPs iteratively and maintaining the scalability at each iteration. In particular, we will combine the basis pursuit idea in \cite{ahmadi2017sum} and the tight approximation quality of block factor-width-two matrices in \cite{zheng2022block}.


\section{Inner approximations of the psd cone} \label{section: inner-outer-approximation}
Given a partition $\alpha$, we know $\mathcal{FW}^{n}_{\alpha,2} \subseteq \mathbb{S}^n_+ \subseteq (\mathcal{FW}^{n}_{\alpha,2})^*$. Then, replacing $\mathbb{S}^n_+$ with $\mathcal{FW}^{n}_{\alpha,2}$ (or $(\mathcal{FW}^{n}_{\alpha,2})^*$, respectively) in \eqref{eq:SDPs} naturally gives an inner (outer, respectively) approximation for solving SDPs~\cite{zheng2022block}. In this section, motivated by the basis  pursuit idea in \cite{ahmadi2017sum}, we introduce an iterative algorithm  for inner approximations of \eqref{eq:SDPs}. Our algorithm returns a sequence of upper bounds with increasing accuracy.    



\subsection{Inner Approximations} \label{sec:inner}
For the inner approximation, we start from replacing the PSD constraint in \eqref{eq:SDP-primal} by $\mathcal{FW}^{n}_{\alpha,2}$, leading to  
\begin{subequations}  \label{eq:inner-approximation-step-0}
    \begin{align}
      {\mathsf U}_{\alpha}^1 :=  \min_{X} \quad & \langle C, X\rangle \nonumber \\
        \mathrm{subject~to} \quad & \langle A_i, X\rangle = b_i, \quad i = 1, \ldots, m,  \label{eq:inner-approximation-step-0-con1}\\
        & X \in \mathcal{FW}^{n}_{\alpha,2}, \label{eq:inner-approximation-step-0-con2}
    \end{align}
\end{subequations}
which provides an upper bound for \eqref{eq:SDP-primal}. 
The cyclic property of the trace operator leads to
$$
\begin{aligned}
\langle C, X\rangle &= \left\langle C, \sum_{1 \leq k<l \leq p}
        (E_{kl}^{\alpha})^\tr X_{kl}E_{kl}^{\alpha}\right\rangle \\
        &= \sum_{1 \leq k<l \leq p} \left\langle E_{kl}^{\alpha} C(E_{kl}^{\alpha})^\tr, 
         X_{kl}\right\rangle. 
\end{aligned}
$$
This allows us to equivalently rewrite~\eqref{eq:inner-approximation-step-0} into
\begin{equation} \label{eq:inner-approximation-step-1}
    \begin{aligned}
        {\mathsf U}_{\alpha}^1 := \min_{X_{kl}} \quad & \sum_{1 \leq k<l \leq p} \left\langle C_{kl}, 
         X_{kl}\right\rangle \\
        \mathrm{subject~to} \quad & \sum_{1 \leq k<l \leq p} \left\langle A_{i,kl}, 
         X_{kl}\right\rangle = b_i, i = 1, \ldots, m,  \\
        & \;\; X_{kl} \in \mathbb{S}^{\alpha_k +\alpha_l}_+, \qquad \quad   1 \leq k < l \leq p,
    \end{aligned}
\end{equation}
where $C_{kl} = E_{kl}^{\alpha} C(E_{kl}^{\alpha})^\tr, A_{i,kl} = E_{kl}^{\alpha} A_i(E_{kl}^{\alpha})^\tr, 1 \leq k < l \leq p,i = 1, \ldots, m$. We can now use standard conic solvers (such as SeDuMi \cite{sturm1999using} and MOSEK \cite{aps2019mosek}) to solve \eqref{eq:inner-approximation-step-1}. This gives an upper bound 
\begin{equation}
    d^{\star} =  p^{\star} \leq {\mathsf U}_{\alpha}^1. 
\end{equation}
The gap ${\mathsf U}_{\alpha}^1 - p^{\star} $ may be large. By Theorem \ref{theo:block-factor-width-two}, using a coarser partition $ \alpha \sqsubseteq \beta$ can reduce the gap ${\mathsf U}_{\beta}^1 - p^{\star} \leq  {\mathsf U}_{\alpha}^1 - p^{\star} $, but this leads to an SDP with a larger PSD constraint in \eqref{eq:inner-approximation-step-1}.  

We introduce another way to reduce the gap by solving a sequence of SDPs in the form of \eqref{eq:inner-approximation-step-1} while keeping the same partition $\alpha$. In particular, given an $n \times n$ matrix $V$, we define a family of cones
\begin{equation}\label{eq:family-of-cone}
    \begin{aligned}
       \mathcal{FW}^n_{\alpha,2}(V) \!\coloneqq \! \{M\in \mathbb{S}^{n} \mid M \!=\!V^\tr QV,\, Q \!\in\! \mathcal{FW}^n_{\alpha,2} \}. 
    \end{aligned}
\end{equation}
 It is clear that $\mathcal{FW}^n_{\alpha,2}(V) = \mathcal{FW}^n_{\alpha,2}$ when $V = I$, and that $ \mathcal{FW}^n_{\alpha,2}(V)$ is an inner approximation of $\mathbb{S}^n_+$ for any $V$.

When $V$ is fixed, linear optimization over $\mathcal{FW}^n_{\alpha,2}(V)$ amounts to solve an SDP in a similar form to \eqref{eq:inner-approximation-step-1}.
In particular, at each iteration $t$, we replace $\mathcal{FW}^n_{\alpha,2}$ in \eqref{eq:inner-approximation-step-0} with $\mathcal{FW}^n_{\alpha,2}(V_t)$, and get the following problem
\begin{equation} \label{eq:inner-approximation-step-k}
    \begin{aligned}
        {\mathsf U}_{\alpha}^t := \min_{X_{kl}} ~ & \sum_{1 \leq k<l \leq p} \left\langle \hat{C}_{kl}, 
         X_{kl}\right\rangle \\
        \mathrm{subject~to} ~& \sum_{1 \leq k<l \leq p} \left\langle \hat{A}_{i,kl}, 
         X_{kl}\right\rangle = b_i, i = 1, \ldots, m  \\
        & \;\; X_{kl} \in \mathbb{S}^{\alpha_k +\alpha_l}_+, \qquad    1 \leq k < l \leq p,
    \end{aligned}
\end{equation}
where the problem data are
\begin{equation}\label{eq:update}
  \begin{aligned}
\hat{C}_{kl} &= E_{kl}^{\alpha}\left(V_t CV_t^\tr\right)(E_{kl}^{\alpha})^\tr, \\
\hat{A}_{i,kl}& = E_{kl}^{\alpha} \left(V_t A_iV_t^\tr\right)(E_{kl}^{\alpha})^\tr,\quad 1 \leq k < l \leq p.
\end{aligned}  
\end{equation}

%
%
%
%
We choose the sequence of matrices $\{V_t\}$ as 
\begin{equation} \label{eq:matrices_U}
    \begin{split}
         V_1 & = I\\
         V_{t+1} & = \mathrm{chol}(X_{t}^{\star}),  
    \end{split}
\end{equation}
where $\mathrm{chol}(\cdot)$ denotes a Cholesky factorization, and $X_{t}^{\star} := \sum_{1 \leq k<l \leq p}
        V_{t}^\tr(E_{kl}^{\alpha})^\tr X_{kl}^{t,\star}E_{kl}^{\alpha}V_{t}$ is the optimal solution to \eqref{eq:inner-approximation-step-k} at iteration $t$\footnote{We assume that the first iteration is feasible. This guarantees the feasibility of the rest of iterations.}. When choosing $V_1 = I$ at iteration 1, problem \eqref{eq:inner-approximation-step-k} reduces to \eqref{eq:inner-approximation-step-1}. 

\subsection{Monotonically decreasing upper bounds}
The choice of the matrices $V_{t+1}$ as the factorization of $X_{t}^{\star}$ in \eqref{eq:matrices_U} leads to a sequence of monotonically decreasing cost values in \eqref{eq:inner-approximation-step-k}. We have the following proposition.

\begin{proposition}\label{prop:decreasing}
    Given any partition $\alpha$, solving \eqref{eq:inner-approximation-step-k} with matrices $\{V_{t}\}$ in \eqref{eq:matrices_U} leads to 
    \vspace{-1.5mm}
    $${\mathsf U}^{1}_\alpha \geq {\mathsf U}^{2}_\alpha \geq \ldots \geq {\mathsf U}^{t}_\alpha \geq {\mathsf U}^{t+1}_\alpha \geq p^\star.$$
\end{proposition}
\begin{proof}
Upon choosing $V_{t+1} = \mathrm{chol}(X_t^{\star})$, we naturally have $X_{t} = V_{t+1}^\tr \times I \times V_{t+1}$. Since $I \in \mathcal{FW}^n_{\alpha,2}$, we have $X_{t}^{\star} \in \mathcal{FW}^n_{\alpha,2}(V_{t+1})$. It means that the optimal solution $X_{t}^{\star}$ at iteration $t$ is in the feasible region of the SDP at iteration $t+1$. Thus, we have ${\mathsf U}^{t}_\alpha \geq {\mathsf U}^{t+1}_\alpha$. 
\end{proof}

\begin{algorithm}[t]
\SetKwData{Left}{left}\SetKwData{This}{this}\SetKwData{Up}{up}
\SetKwFunction{Union}{Union}\SetKwFunction{FindCompress}{FindCompress}

\caption{Inner-approximations using~$\mathcal{FW}^n_{\alpha,2}$}
\KwIn{SDP data $A_i,C \in \mathbb{S}^n, b\in \mathbb{R}^m$, block partition $\alpha$, and maximum iteration $t_{\max}$}
\KwOut{Upper bound ${\mathsf U}_{\alpha}$}
Initialize $t=1;$ $V_{1}= I;$ \\
    \While{$t<t_{\max}$ }
    {
    Solve \eqref{eq:inner-approximation-step-k} to get ${\mathsf U}_{\alpha}^t$ and $X_{kl}^{t,\star}, 1 \leq k<l\leq p;$ \\
    Set ${\mathsf U}_{\alpha} = {\mathsf U}_{\alpha}^t;$\\
    Compute
    $$ 
    \begin{aligned}
    & X_t^{\star} = \sum_{1 \leq k<l \leq p}
        V_{t}^\tr(E_{kl}^{\alpha})^\tr X_{kl}^{t,\star}E_{kl}^{\alpha}V_{t};
        \\ & V_{t+1} = \mathrm{chol}(X_t^{\star});
    \end{aligned}
    $$ \\
    Update $\hat{C}_{kl},\hat{A}_{i,kl}$ as \eqref{eq:update}$;$
    Set $t=t+1;$ 
    }
    \Return {${\mathsf U}_{\alpha}$ }
\label{Algorithm:inner approximation}
\end{algorithm}

When $X_t^{\star}$ is positive definite, we have a strictly decreasing cost value, as summarized in the following theorem.

\begin{theorem} \label{theorem:strictly decreasing}
    Given any partition $\alpha$, let $X_t^{\star}$ be an optimal solution of \eqref{eq:inner-approximation-step-k} at iterate $t$. If $X_t^{\star}$ is positive definite and $ {\mathsf U}^{t}_{\alpha} > p^\star$, then ${\mathsf U}^{t}_{\alpha} > {\mathsf U}^{t+1}_{\alpha} \geq p^\star $.
\end{theorem}
\begin{proof}
Let $X^{\star}$ and $p^{\star}$ be the optimal solution and cost value of \eqref{eq:SDPs}. We construct a point 
\begin{equation} \label{eq:Xhat}
  \hat{X}\coloneqq (1-\lambda)X_t^{\star}+\lambda X^{\star}, 
\end{equation}
with some $\lambda \in (0,1)$. We will prove there exists a $\lambda \in (0,1)$ such that $ \hat{X}$ in \eqref{eq:Xhat} is feasible for \eqref{eq:inner-approximation-step-k} at iteration $t + 1$. Therefore, we complete the proof by observing  
$$
U^{t+1}_{\alpha} \leq \langle C, \hat{X}\rangle = (1-\lambda)\langle C, X_t^{\star} \rangle+\lambda\langle C,X^{\star} \rangle < U^{t}_{\alpha}.
$$

To prove $\hat{X}$ is feasible at iteration $t+1$ for some $\lambda \in (0,1)$, we need to show $\hat{X}$ satisfies 
\begin{itemize}
    \item the linear constraint \eqref{eq:inner-approximation-step-0-con1}, 
    \item the conic constraint \eqref{eq:inner-approximation-step-0-con2} with $\mathcal{FW}^n_{\alpha,2}(V_{t+1})$. 
\end{itemize}
First, it is clear that  both $X_t^{\star}$ and $X^{\star}$ satisfy \eqref{eq:inner-approximation-step-0-con1}, i.e., 
$$
\begin{aligned}
    \langle A_i, X_t^{\star} \rangle &= b_i,\;\; i = 1, \ldots, m,\\
    \langle A_i, X^{\star} \rangle &= b_i,\;\; i = 1, \ldots, m.
\end{aligned}
$$
Then, we have $\forall i=1, \ldots, m$,
$$
\begin{aligned}
\left\langle A_i, \hat{X} \right\rangle &= \langle A_i, (1-\lambda)X_t^{\star}+\lambda X^{\star}\rangle \\
&= (1-\lambda) \langle A_i,X_t^{\star} \rangle +\lambda \langle A_i, X^{\star} \rangle = b_i.
\end{aligned}
$$ 
Since $X_t^{\star}=V_{t+1}^\tr V_{t+1}$ and $X_t^{\star}$ is positive definite, $V_{t+1}$ must be invertible. We let 
$$
\tilde{X}_{t+1}  \coloneqq \left(V_{t+1}^{-1}\right)^{\tr}X^{\star}V_{t+1}^{-1}.
$$ 
Then, for small enough $\lambda >0 $, the matrix $(1-\lambda)I + \lambda \tilde{X}_{t+1} \in \mathcal{FW}^{n}_{\alpha,2}$. Hence, from \eqref{eq:Xhat}, we have 
$$\hat{X} = V^{\tr}_{t+1}((1-\lambda)I + \lambda \tilde{X}_{t+1})V_{t+1} \in \mathcal{FW}^{n}_{\alpha,2}(V_{t+1}).$$
This completes the proof.
\end{proof}

Our proof is inspired by \cite[Theorem 3.1]{ahmadi2017sum}, and we generalize it to any block partition $\alpha$, including the iterative algorithm based on SDD matrices \cite[Section 4]{ahmadi2017sum} as a special case. 
 The key idea is to make sure the optimal solution of the previous iteration is a feasible point in the next iteration. Thus, instead of the Cholesky decomposition, we can use other choices, such as spectral decomposition $X_t^{\star} = PDP^\tr$. 

Algorithm \ref{Algorithm:inner approximation} lists the overall procedure of the proposed iterative inner approximations for solving~\eqref{eq:SDPs}. We use a simple example to illustrate our algorithm. 

\begin{figure}[t]
    \centering
    \setlength{\abovecaptionskip}{2pt}
    \includegraphics[width=0.45\textwidth]{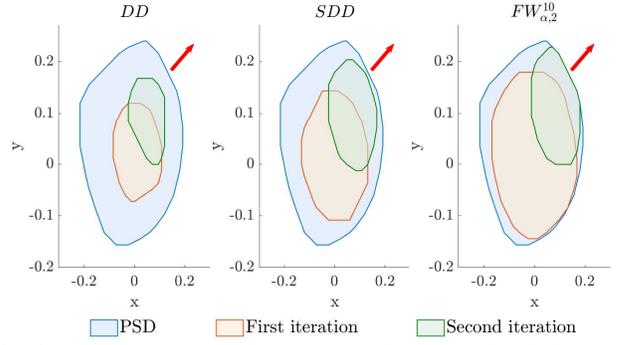}
    \caption{Feasible regions of inner approximations of  \eqref{eq:inner-example} in Algorithm \ref{Algorithm:inner approximation} using $DD$, $SDD$, and $\mathcal{FW}^{n}_{\alpha, 2}$ with $\alpha = \{2,2,2,2,2\}$. The red arrows denote the decreasing  direction of the cost value.   }
    \label{Fig:Inner_Approx}
\end{figure}

\begin{table}[t]
  \begin{center}
    \caption{Cost values of inner approximations  in Algorithm \ref{Algorithm:inner approximation} for solving \eqref{eq:inner-example}. We used $DD$, $SDD$, and $\mathcal{FW}^{n}_{\alpha, 2}$ with $\alpha = \{2,2,2,2,2\}$. The optimal cost of \eqref{eq:inner-example} is $-0.298$.}
    \label{table:inner}
    \small 
    \begin{tabular}{c p{9mm}p{8mm}p{9mm}p{8mm}p{9mm}p{8mm}}
    \toprule
    &\multicolumn{2}{c}{DD}&  \multicolumn{2}{c}{SDD} & \multicolumn{2}{c}{$\mathcal{FW}^n_{\alpha,2}$}  \\
    \cmidrule(lr){2-3}
    \cmidrule(lr){4-5}
    \cmidrule(lr){6-7}
    Iter & Cost & Gap & Cost & Gap & Cost & Gap\\
    \toprule
    $1$ & $-0.148$ & $50.3$\% & $-0.176$ & $40.9$\% & $-0.232$ & $22.2$\%\\
    $2$ & $-0.236$ & $20.8$\% & $-0.277$ & $7.04$\% & $-0.298$ & $0$\%\\
    \toprule
    \end{tabular}
  \end{center}
\end{table}

\begin{example}
Consider an SDP of the form
\begin{equation} \label{eq:inner-example}
    \begin{aligned}
        \min_{x,y} \quad & -x-y \\
        \mathrm{subject~to} \quad & I+xA+yB \succeq 0,
    \end{aligned}
\end{equation}
where $A$ and $B$ are two $10 \times 10$ matrices with each entry randomly generated. We consider inner approximations by DD, SDD and $\mathcal{FW}_{\alpha}^{10}$ with $\alpha=\{2,2,2,2,2\}$. The results are shown in Figure \ref{Fig:Inner_Approx}. 
The blue part in Figure \ref{Fig:Inner_Approx} shows the feasible region of \eqref{eq:inner-example}. We then replace the semidefinite constraint by $DD$, $SDD$ and $\mathcal{FW}^{10}_{\alpha, 2}$. 
Orange and green parts in Figure \ref{Fig:Inner_Approx} show the feasible regions in iteration 1 and 2 of Algorithm~\ref{Algorithm:inner approximation}. It is clear that the feasible region moves towards~the direction where the cost decreases. 
As shown in Table \ref{table:inner} (also in Figure \ref{Fig:Inner_Approx}),  our algorithm using $\mathcal{FW}^{n}_{\alpha, 2}$ achieves the optimal cost at the second iteration, while the results from DD/SDD approximations~\cite{ahmadi2017sum} are still far away from the optimal cost. 
\hfill $\square$
\end{example}



\section{Outer Approximations of the PSD cone} \label{sec:Outer}

The inner approximation in \eqref{eq:inner-approximation-step-0} provides an upper bound of the SDPs \eqref{eq:SDPs}. 
Here, we introduce an outer approximation for the same problem \eqref{eq:SDP-primal}, 
which provides a lower bound. Therefore, the optimal cost of \eqref{eq:SDPs} can be bounded above and below simultaneously. 

\subsection{Outer approximations} 
Consider the relationship $\mathcal{FW}^{n}_{\alpha,2} \subseteq \mathbb{S}^{n}_{+} \subseteq (\mathcal{FW}^{n}_{\alpha,2})^{*}$. Replacing PSD cone $\mathbb{S}^n_+$ by the dual cone $(\mathcal{FW}^{n}_{\alpha,2})^{*}$ gives an outer approximation of \eqref{eq:SDP-primal}, i.e.,
\begin{equation} \label{eq:outer-approximation-cone}
    \begin{aligned}
        {\mathsf L}_{\alpha}^1 \coloneqq \min_{X} \quad & \langle C, X\rangle \\
        \mathrm{subject~to} \quad & \langle A_i, X\rangle = b_i,\qquad i = 1, \ldots, m,  \\
        & X \in (\mathcal{FW}^{n}_{\alpha,2})^*.
    \end{aligned}
\end{equation}
We have 
$
    {\mathsf L}_{\alpha}^1 \leq d^\star = p^\star. 
$
Note that the dual cone $(\mathcal{FW}^{n}_{\alpha,2})^*$ admits a decomposition as \cite{zheng2022block}
\begin{equation}
    \begin{aligned}
        (\mathcal{FW}^{n}_{\alpha,2})^* = \Big\{X \in \mathbb{S}^{n} \mid &E_{kl}^\alpha  X (E_{kl}^\alpha)^\tr \in \mathbb{S}^{\alpha_k + \alpha_l}_+, \\
        & \qquad \qquad 1 \leq k < l \leq p \Big\}. 
        \end{aligned}
\end{equation}
Therefore, problem \eqref{eq:outer-approximation-cone} can be rewritten as:
\begin{equation} \label{eq:outer-approximation-step-0}
    \begin{aligned}
        {\mathsf L}_{\alpha}^1 = \min_{X} \quad & \langle C, X\rangle \\
        \mathrm{subject~to} \quad & \langle A_i, X\rangle = b_i,\qquad i = 1, \ldots, m,  \\
        & E_{kl}^\alpha  X (E_{kl}^\alpha)^\tr \in \mathbb{S}^{\alpha_k + \alpha_l}_+, 1 \leq k < l \leq p.
    \end{aligned}
\end{equation}

The gap $p^\star - {\mathsf L}_{\alpha}^1 $ might be large. 
Similar to inner approximations, we aim to solve a sequence of outer approximations in the following form 
%
    \begin{align}
        {\mathsf L}_{\alpha}^t \coloneqq \min_{X} \quad & \langle C, X\rangle \nonumber \\
        \mathrm{subject~to} \quad & \langle A_i, X\rangle = b_i,\qquad i = 1, \ldots, m,  \label{eq:outer-approximation-step-t} \\
        & E_{kl}^{\alpha}V_tXV_t^\tr (E_{kl}^\alpha)^\tr \in \mathbb{S}^{\alpha_k + \alpha_l}_+, 1 \leq k < l \leq p,\nonumber 
    \end{align}
which is parameterized by $V_t \in \mathbb{R}^{n \times n}$. 
However, we cannot generate the matrix $V_{t+1}$ by Cholesky decomposition of the optimal solution $X_{t}$ of \eqref{eq:outer-approximation-step-t}, since it is not positive semidefinite. To resolve this, motivated by~\cite{ahmadi2017sum}, we look into the dual problem of \eqref{eq:outer-approximation-step-t}, which is 
%
%
    \begin{align}
        {\mathsf L}_{\alpha}^t =  \max_{y,X_{kl}} \quad &  b^\tr y \nonumber \\
        \mathrm{subject~to} \quad & C-\sum_{i=1}^{m}y_iA_i =\sum_{1 \leq k<l \leq p} V_t^\tr (E_{kl}^{\alpha})^\tr X_{kl}E_{kl}^{\alpha} V_t, \nonumber\\
        & X_{kl} \in \mathbb{S}^{\alpha_k + \alpha_l}_{+}, \quad  1 \leq k < l \leq p. \label{eq:outer-approximation-step-t-dual}
    \end{align}
For the optimal solution $y^{t, \star}$ of \eqref{eq:outer-approximation-step-t-dual} at iteration $t$,~the matrix $C-\sum_{i=1}^{m}y_i^{t,\star}A_i$ is guaranteed to be positive semidefinite. Then, we choose a sequence of matrices $\{V_t\}$ for \eqref{eq:outer-approximation-step-t-dual} as
\begin{equation}  \label{eq:outer-approximation-Vt}
    \begin{split}
        V_1 & = I\\
        V_{t+1} & = \mathrm{chol}\left(C-\sum_{i=1}^{m}y_i^{t, \star}A_i\right), 
    \end{split}
\end{equation}
where $y^{t, \star}$ is the optimal solution of \eqref{eq:outer-approximation-step-t-dual} at iteration $t$.

\subsection{Monotonically increasing lower bounds}


The lower bounds from the sequence of outer approximations defined in \eqref{eq:outer-approximation-step-t-dual} and \eqref{eq:outer-approximation-Vt} are monotonically increasing, as proved in the following result. 

\begin{proposition} \label{prop:increasing}
    Given any partition $\alpha$, solving \eqref{eq:outer-approximation-step-t-dual} with matrices $\{V_{t}\}$ in \eqref{eq:outer-approximation-Vt} leads to 
    \vspace{-1.5mm}
    $${\mathsf L}^{1}_\alpha \leq {\mathsf L}^{2}_\alpha \leq \ldots \leq {\mathsf L}^{t}_\alpha \leq {\mathsf L}^{t+1}_\alpha \leq d^\star = p^{\star}.$$
\end{proposition}
\begin{proof}
    It is not difficult to see that the dual problem \eqref{eq:outer-approximation-step-t-dual} is equivalent to
    \begin{align}
         \max_{y} \quad &  b^\tr y \nonumber \\
        \mathrm{subject~to} \quad & C-\sum_{i=1}^{m}y_iA_i  \in  \mathcal{FW}^n_{\alpha,2}(V_t).  \label{eq:outer-approximation-step-t-dual-s1}
    \end{align}
    Suppose the optimal solution to \eqref{eq:outer-approximation-step-t-dual-s1} at iteration $t$ is $y_i^{t, \star}$.~Let 
    \[
    Z_t^\star =C-\sum_{i=1}^{m}y_i^{t, \star}A_i, \quad 
    V_{t+1} =  \mathrm{chol} (Z_t^\star).
    \] 
    %
    %
    %
    We have $Z_{t}^\star = V_{t+1}^\tr \times I \times  V_{t+1} \in \mathcal{FW}^n_{\alpha,2}(V_{t+1})$ since $I \in \mathcal{FW}^n_{\alpha,2}$. It means that $y_i^{t,\star}$ is feasible to \eqref{eq:outer-approximation-step-t-dual-s1} at iteration $t+1$. Thus, we have ${\mathsf L}^{t}_\alpha \leq {\mathsf L}^{t+1}_\alpha$. 
\end{proof}

Similar to Theorem \ref{theorem:strictly decreasing}, when $C-\sum_{i=1}^{m}y_i^{t,\star}A_i$ is strictly positive definite, we have a strictly increasing cost, as summarized in the following theorem. 
\begin{theorem} \label{theorem:strictly increasing}
    Given any partition $\alpha$, let $\{y_{i}^{t,\star},X_{kl}^{t,\star}\}$ be an optimal solution of \eqref{eq:outer-approximation-step-t-dual} at iterate $t$. If $C-\sum_{i=1}^{m}y_i^{t,\star}A_i$ is strictly positive definite and  ${\mathsf L}^{t}_\alpha < d^\star$, then ${\mathsf L}^{t}_\alpha < {\mathsf L}^{t+1}_\alpha \leq d^\star$.
\end{theorem}
\begin{proof}
    Let $({y^\star,Z^\star})$ be the optimal solution of \eqref{eq:SDP-dual}, and $y_i^{t, \star}$ is an optimal solution of \eqref{eq:outer-approximation-step-t-dual-s1} at iteration $t$. 
    %
    %
    We construct 
\begin{equation} \label{eq:yhat}
    \hat{y} \coloneqq (1-\lambda)y^{t,\star}+\lambda y^\star,
\end{equation}
    with some $\lambda \in (0,1)$. 
    %
    %
    We will prove there exists a $\lambda \in (0,1)$ such that $ \hat{y}$ in \eqref{eq:yhat} is feasible for \eqref{eq:outer-approximation-step-t-dual-s1} at iteration $t + 1$. 
    We then complete the proof by observing  
    $${\mathsf L}^{t+1}_{\alpha} = b^\tr y^{t+1, \star} \geq b^\tr \hat{y} =(1-\lambda)b^\tr y^{t,\star}+\lambda b^\tr y^\star > {\mathsf L}^{t}_{\alpha}.
    $$

    To prove $ \hat{y}$ is feasible at iteration $t+1$, we need to show 
    \begin{equation} \label{eq:FW-t+1-y}
        C-\sum_{i=1}^{m}\hat{y}_iA_i  \in  \mathcal{FW}^n_{\alpha,2}(V_{t+1}), \;\text{for some}\; \lambda \in (0,1).
    \end{equation}
    %
    Indeed, we have 
    $$
        \begin{aligned}
            &C-\sum_{i=1}^{m}\hat{y}_iA_i \\
            =& (1-\lambda) \left(C- \sum_{i=1}^{m} y^{t,\star}_i A_i\right) + \lambda \left(C- \sum_{i=1}^{m} y^{\star}_i A_i\right).
        \end{aligned}
    $$
    %
    %
    %
    %
    Since $C- \sum_{i=1}^{m} y^{t,\star}_i A_i=V_{t+1}^\tr V_{t+1}$ that is strictly positive definite, $V_{t+1}$ must be invertible. Then, we have 
    $$
    \begin{aligned}
    &\left(V_{t+1}^{-1}\right)^{\tr}\left(C-\sum_{i=1}^{m}\hat{y}_iA_i\right)V_{t+1}^{-1} \\
    = &(1-\lambda) I + \lambda \left(V_{t+1}^{-1}\right)^{\tr}\left(C- \sum_{i=1}^{m} y^{\star}_i A_i\right)V_{t+1}^{-1} \in \mathcal{FW}^n_\alpha
    \end{aligned}
    $$
    when $\lambda \in (0,1)$ is small enough. It means that \eqref{eq:FW-t+1-y} holds with this $\lambda$ for $ \hat{y}$ in \eqref{eq:yhat}. 
    This completes the proof.
\end{proof}

\begin{algorithm}[t]
\caption{Outer-approximations using~$\mathcal{FW}^n_{\alpha,2}$}

\SetKwData{Left}{left}\SetKwData{This}{this}\SetKwData{Up}{up}
\SetKwFunction{Union}{Union}\SetKwFunction{FindCompress}{FindCompress}

\KwIn{SDP data $A_i,C \in \mathbb{S}^n, b\in \mathbb{R}^m$, block partition $\alpha$, and maximum iteration $t_{\max}$ }
\KwOut{Lower bound ${\mathsf L}_{\alpha}$}
Initialize $t=1; V_{1} = I;$ \\
    \While{$t<t_{\max}$ }
    {
    Solve $\eqref{eq:outer-approximation-step-t-dual}$ to get ${\mathsf L}_{\alpha}^t$ and $y^{t, \star};$
    \\Set ${\mathsf L}_{\alpha}={\mathsf L}_{\alpha}^t;$\\
    Compute $V_{t+1} = \mathrm{chol}(C-\sum_{i=1}^{m}y_i^{t, \star}A_i);$\\
    Set $t=t+1;\;$
    }
    \Return {${\mathsf L}_{\alpha}$ }
\label{Algorithm:outer approximation}
\end{algorithm}

\begin{remark}[Solving outer approximations]
Unlike the inner approximations \eqref{eq:inner-approximation-step-k}, the outer approximations \eqref{eq:outer-approximation-step-t} and \eqref{eq:outer-approximation-step-t-dual} are not in the standard form of SDPs. Thus they cannot be solved directly using standard conic solvers. In our implementation, we apply the idea in~\cite{lofberg2009dualize} and transform \eqref{eq:outer-approximation-step-t-dual} into the following primal form of SDPs 
\begin{align}
       \min_{y,X_{kl}} \; &  - b^\tr y \nonumber \\
        \mathrm{subject~to} \; & \sum_{1 \leq k<l \leq p} V_t^\tr (E_{kl}^{\alpha})^\tr X_{kl}E_{kl}^{\alpha} V_t +\sum_{i=1}^{m}y_iA_i = C, \nonumber\\
        & X_{kl} \in \mathbb{S}^{\alpha_k + \alpha_l}_{+},  1 \leq k < l \leq p, \label{eq:outer-approximation-step-t-dual-reform}
    \end{align}
which is ready to be solved using standard conic solvers. We note that the size of PSD constraints has been reduced in \eqref{eq:outer-approximation-step-t-dual-s1}, but the number of equality constraints is $n^2$. Thus, solving \eqref{eq:outer-approximation-step-t-dual-s1} might not be as efficient as solving \eqref{eq:inner-approximation-step-k}.  \hfill $\square$ 
\end{remark}

Our iterative outer approximations for solving \eqref{eq:SDPs} is listed in Algorithm \ref{Algorithm:outer approximation}. 
We use SDP \eqref{eq:inner-example} to illustrate our algorithm. 

\begin{example}



The feasible regions in the first and $11$th~iterations are shown in Figure \ref{Fig:Outer_Approx}. 
In particular, the blue part  shows the feasible region of \eqref{eq:inner-example}. We then replace the PSD constraint by $(DD)^*$, $(SDD)^*$ and $(\mathcal{FW}^{10}_{\alpha, 2})^*$. Orange and green regions are the feasible regions in iterations~1~and~11. It is clear that the feasible region moves towards the direction where the cost increases. 
As shown in  Table \ref{table:outer} (also in Figure \ref{Fig:Outer_Approx}),  our algorithm using $\mathcal{FW}^{n}_{\alpha, 2}$ achieves the optimal cost at iteration 11, while the results from DD/SDD approximations~\cite{ahmadi2017sum} are still far away from the optimal cost.  
\hfill $\square$
\end{example} 

Figure \ref{fig: inner/outer-approximation-simple-example-plot} shows the convergence of the upper and lower bounds of SDP \eqref{eq:inner-example} from Algorithm \ref{Algorithm:inner approximation} and \ref{Algorithm:outer approximation}. In this case, the convergence using $\mathcal{FW}^{n}_{\alpha, 2}$ is much faster than the DD/SDD strategies \cite{ahmadi2019dsos,ahmadi2017sum}. 

\begin{figure}
    \centering
     \setlength{\abovecaptionskip}{1pt}
        \includegraphics[width=0.45\textwidth]{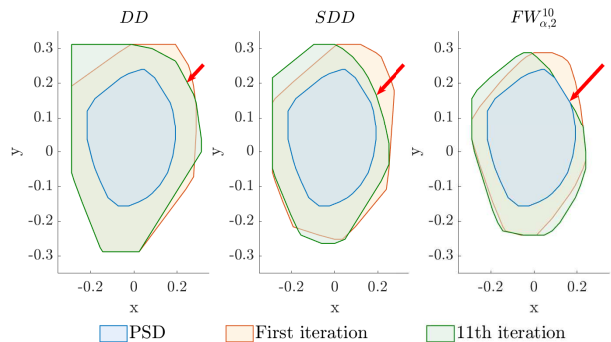}
    \caption{Feasible regions of outer approximations in Algorithm~\ref{Algorithm:outer approximation} by $DD$, $SDD$, and $\mathcal{FW}^{n}_{\alpha, 2}$ with $\alpha=\{2,2,2,2,2\}$. The red arrow denotes the increasing direction of the cost value.   }
    \label{Fig:Outer_Approx}
\end{figure}



\begin{table}[t]
  \begin{center}
    \caption{Cost values of iteratively outer approximation  \eqref{eq:inner-example} using $DD$, $SDD$, and $\mathcal{FW}^{n}_{\alpha, 2}$. The optimal cost value of \eqref{eq:inner-example} is $-0.298$.}
    \label{table:outer}
    \small
    \begin{tabular}{c p{9mm}p{8mm}p{9mm}p{8mm}p{9mm}p{8mm}}
    
    \toprule
    &\multicolumn{2}{c}{DD}&  \multicolumn{2}{c}{SDD} & \multicolumn{2}{c}{$\mathcal{FW}^n_{\alpha,2}$}  \\
    \cmidrule(lr){2-3}
    \cmidrule(lr){4-5}
    \cmidrule(lr){6-7}
    Iter & Cost & Gap & Cost & Gap & Cost & Gap\\
    \toprule
    $1$ & $-0.499$ &$67.5$\% & $-0.469$ &$57.4$\% & $-0.401$ & $34.6$\% \\
    $11$ & $-0.443$ &$48.7$\% & $-0.350$ &$17.6$\% & $-0.298$ & $0$\%\\
    \toprule
    \end{tabular}
  \end{center}
\end{table}

\begin{figure}[t] 
    \centering
    \includegraphics[width=0.32\textwidth]{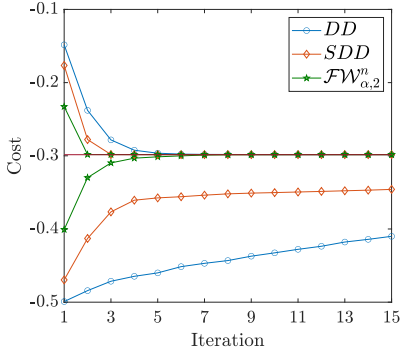}
        \caption{Inner/Outer approximations of SDP \eqref{eq:inner-example} using DD, SDD, and $\mathcal{FW}^n_{2,\alpha}$ with $\alpha=\{2,\ldots,2\}$.}
    \label{fig: inner/outer-approximation-simple-example-plot}
\end{figure}

\begin{remark}[Role of partition $\alpha$] \label{remark:partition}
In our Algorithms \ref{Algorithm:inner approximation}-\ref{Algorithm:outer approximation}, the choice of partition $\alpha$ brings  flexibility in balancing the computational efficiency and solution quality at each iteration. Choosing a suitable partition might be problem dependent; we refer interested readers to \cite{zheng2022block} for more discussions. Here, we highlight that 1) a coarser partition normally leads to faster convergence in Algorithms \ref{Algorithm:inner approximation}-\ref{Algorithm:outer approximation}, as shown in our extensive numerical experiments in Section \ref{section:Numerical Results}; 2)  a coarser partition also leads to a smaller number $p$ in \eqref{eq:outer-approximation-step-t-dual} and \eqref{eq:inner-approximation-step-k}. The latter fact is important in constructing the problem in each iteration, especially for large-scale cases. For example, when $n = 2000$, if we use the SDD matrices for inner/outer approximations \cite{ahmadi2019dsos,ahmadi2017sum}, the number of small blocks is $\binom{2000}{2} = 1\,999\,000$ which is too large to even construct the problem instances \eqref{eq:outer-approximation-step-t-dual} and \eqref{eq:inner-approximation-step-k}. We indeed failed to construct such problems in our experiments in Section~\ref{subsection:random-SDP}. Instead, for $\alpha = \{10,\ldots,10\}$ ($\beta = \{20,\ldots,20\}$, respectively), the number of blocks is reduced to $ \binom{200}{2}=19900$ ($\binom{100}{2}=4950$, respectively), for which efficient constructions exist.


\end{remark}

\section{Numerical Results} \label{section:Numerical Results}
In this section, we present computational results of Algorithms \ref{Algorithm:inner approximation}-\ref{Algorithm:outer approximation} on three classes of SDPs: 1) independent stable set, 2) binary quadratic optimization, and 3) SDPs with random data.
Our experiments were carried out  in MATLAB R2021a on a Windows PC with 2.6 GHz speed and 24 GB RAM. All the SDP instances at each iteration of Algorithms \ref{Algorithm:inner approximation}-\ref{Algorithm:outer approximation} were solved by MOSEK \cite{aps2019mosek}.

\begin{figure}[t] 
    \centering
    \setlength{\abovecaptionskip}{1pt}
    \includegraphics[width=0.32\textwidth]{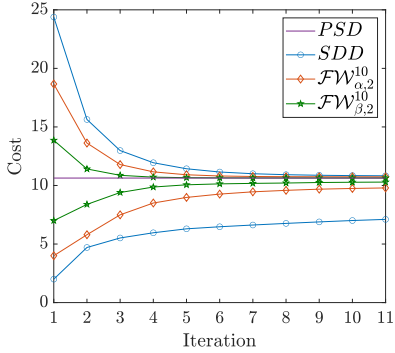}
        \caption{Inner/Outer approximations of Lov$\acute{a}$sz theta number \eqref{eq:stableset_sdp} by different partitions: $\alpha=\{2,\ldots,2\}$, and $\beta=\{5,\ldots,5\}$.}
    \label{Fig:stableset}
\end{figure}

\begin{table}[h]
  \begin{center}
  \setlength{\abovecaptionskip}{1pt}
    \caption{Success rate of upper bounds of $\vartheta(\mathcal{G})$ in \eqref{eq:stableset_sdp} for $140$ instances of 30-nodes Erd$\ddot{o}$s-$\acute{R}$enyi graphs using Algorithm \ref{Algorithm:outer approximation}, where $\alpha=\{2,\ldots,2\}$ and $\beta=\{5,\ldots,5\}$. }
    \label{table:stableset}
    \small
    \begin{tabular}{c r r r}
    \toprule
    Iteration $t$ & SDD & $\mathcal{FW}^n_{\alpha,2}$ & $\mathcal{FW}^n_{\beta,2}$ \\
    \toprule
    $1$ & $0 \%$ & $0 \%$ & $0 \%$ \\
    $3$ & $0 \%$ & $5.7 \%$ & $36.4 \%$  \\
    $5$ & $20 \%$ & $36.4 \%$ & $86.4 \%$  \\
    $7$ & $35.7 \%$ & $79.3 \%$ & $95.7 \%$  \\
    \toprule    
  
    \end{tabular}
  \end{center}
\end{table}

\subsection{The maximum stable set problem}
The maximum stable set problem is a classical combinatorial problem, which aims to find the stability number of a graph.  A \textit{stable set} of a undirected graph $\mathcal{G}=(\mathcal{V},\mathcal{E})$ is a set of nodes of $\mathcal{G}$ such that there are no edges between them. The maximum stable number of $\mathcal{G}$, denoted as $\alpha(\mathcal{G})$, is the size of maximum stable set. 
However, testing whether a $\alpha(\mathcal{G})$ is greater than an integer $k$ is well-known to be NP-complete~\cite{karp1972reducibility}. 
%
%
This problem can be formulated as 
\begin{equation} \label{eq:stabelset}
    \begin{aligned}
        \alpha(\mathcal{G}) \coloneqq  \max_{x_i} \quad &  \sum_{i=1}^{n}x_i \\
        \mathrm{subject~to} \quad & x_i \in \{0,1\}, \quad i = 1,2,\ldots,n,\\
        & x_ix_j = 0, \quad \forall(i,j)\in \mathcal{E}.
    \end{aligned}
\end{equation}

A well-known SDP-based upper bound, introduced in \cite{lovasz1979shannon}, can be computed by
\begin{equation}\label{eq:stableset_sdp}
    \begin{aligned}
        \vartheta(\mathcal{G}) \coloneqq  \max_{X} \quad &  \langle J,X \rangle\\
        \mathrm{subject~to} \quad & \langle I,X\rangle = 1,\\
        & X_{ij}=0, \quad \forall(i,j)\in \mathcal{E},\\
        & X \succeq 0,\\
    \end{aligned}
\end{equation}
%
%
%
where $J$ is an all-one matrix and $I$ is the identity matrix. The cost of \eqref{eq:stableset_sdp} is called Lov$\acute{a}$sz theta number, denoted as  $\vartheta(\mathcal{G})$, which provides an upper bound $\vartheta(\mathcal{G})\geq\alpha(\mathcal{G})$. We now apply Algorithms \ref{Algorithm:inner approximation} and \ref{Algorithm:outer approximation} to get a sequence of upper and lower bounds on Lov$\acute{a}$sz theta number.

We first generated a Erd$\ddot{o}$s-$\acute{R}$enyi graph of 30 nodes with edge probability 0.2, and then applied Algorithms \ref{Algorithm:inner approximation} and \ref{Algorithm:outer approximation} using three different partitions: SDD (trivial partition), $\alpha=\{2,\ldots,2\}$, and $\beta=\{5,\ldots,5\}$. As shown in Figure \ref{Fig:stableset}, a coarser partition $\beta$ leads to the fastest convergence for both inner and outer approximation in this case. 
%
To give a more quantitative comparison, we further generated 140 instances of 30-node Erd$\ddot{o}$s-$\acute{R}$enyi graphs with edge probability from 0.2 to 0.8. We use Algorithm \ref{Algorithm:outer approximation} to compute the upper bound of $\vartheta(\mathcal{G})$. When the upper bound is within 99\% suboptimality to $\vartheta(\mathcal{G})$, we consider it as a success. Table \ref{table:stableset} lists the success rate at different iterations of Algorithm \ref{Algorithm:outer approximation}. As expected, a coarser partition $\beta$ gives much higher success rates compared to SDD approximation~\cite{ahmadi2017sum}. Specifically, in the seventh iteration, $\mathcal{FW}^{n}_{\beta,2}$ obtains $95.7\%$ success rate, while SDD only has $35.7\%$ success rate.


\begin{figure}[t]
    \centering
    \setlength{\abovecaptionskip}{1pt}
    \includegraphics[width=0.32\textwidth]{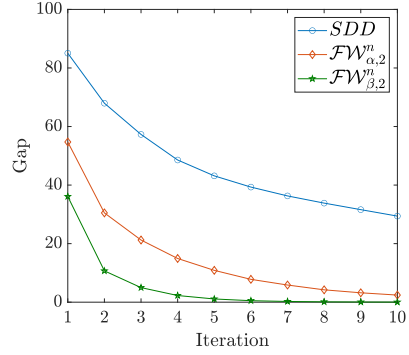}
        \caption{Inner approximation on the binary quadratic optimization \eqref{eq:BQO-sdp-relax} with $\alpha=\{10,\ldots,10\}$ and $\beta=\{20,\ldots,20\}$.}
        \label{Fig:BQO}
\end{figure}

\begin{table}[t]
  \begin{center}
  \setlength{\abovecaptionskip}{1pt}
    \caption{Upper bounds of \eqref{eq:BQO-sdp-relax} at iteration 10 using SDD, $\mathcal{FW}^n_{\alpha,2}$, $\mathcal{FW}^n_{\beta,2}$ with $\alpha=\{10,\ldots,10\},~\beta=\{20,\ldots,20\}$.}
    \label{table:BQO}
    \small
    \begin{tabular}{c*{6}{l}}
    \toprule
    & \multicolumn{2}{c}{SDD} & \multicolumn{2}{c}{$\mathcal{FW}^n_{\alpha,2}$} &\multicolumn{2}{c}{$\mathcal{FW}^n_{\beta,2}$}\\
    \cmidrule(lr){2-3}\cmidrule(lr){4-5}\cmidrule(lr){6-7}
    Instance&Gap&Time& Gap& Time& Gap &Time\\
    \toprule
    1 & $28.77$ &$43.5$ & $3.02$ &$27.0$ & $0.05$& $32.4$\\
    2 & $30.81$ &$51.4$ & $2.31$ &$25.6$ &$0.03$&$30.4$\\
    3 & $29.24$ & $48.4$& $3.54$ &$27.5$ &$0.04$&$29.2$\\
    4 & $29.45$ &$47.2$ & $3.14$ &$26.5$ &$0.04$&$32.7$\\
    5 & $29.40$ & $55.8$& $2.43$ &$26.7$ &$0.03$&$36.9$\\
    \toprule
    \end{tabular}
  \end{center}
\end{table}

\subsection{Binary quadratic optimization}
Binary quadratic optimization is another classical combinatorial problem, in which we have a quadratic cost and a binary decision variable $x$. Formally, the problem is 
\begin{equation} \label{eq:BQO}
    \begin{aligned}
        \min_{x} \quad & x^TQx \\
        \mathrm{subject~to} \quad & x_i^2=1,\quad i = 1,2,\ldots,n,
    \end{aligned}
\end{equation}
where $Q \in \mathbb{S}^{n}$. Many well-known problems are in the form of \eqref{eq:BQO}, such as max-cut problems \cite{festa2002randomized}. 
%
%
The quadratic constraint define a finite set. The number of elements grows with the rate $2^n$. It is well-known that such a problem is NP-complete~\cite{hartmanis1982computers}. 
A standard semidefinite relaxation is  


\begin{equation}
    \begin{aligned}
        \min_{X} \quad & \langle Q, X \rangle \\
        \mathrm{subject~to} \quad &  X_{ii} = 1, \quad i =1,2,\ldots,n, \\ 
        & X \in \mathbb{S}^{n}_+,
    \end{aligned} \label{eq:BQO-sdp-relax}
\end{equation}
which returns a lower bound for \eqref{eq:BQO}. 

We apply Algorithm \ref{Algorithm:inner approximation} for inner approximations  of~\eqref{eq:BQO-sdp-relax}.
We generated five random cost matrices $Q \in \mathbb{S}^{100}$. We apply three different partitions:
SDD, $\alpha = \{10,\ldots,10\}$, and $\beta = \{20,\ldots,20\}$. We set the maximum iteration as $10$. The results are shown in Figure \ref{Fig:BQO} and Table \ref{table:BQO}. As expected again, as the partition becomes coarser, the approximation quality increases. For example, inner approximation by partition $\beta$ obtains almost optimal solution ($\geq 99.9\%$ optimality) within $10$ iterations, while SDD can only achieve around $70 \%$ optimality while taking a longer time (the longer time consumption is related to the large number of small blocks; see Remark \ref{remark:partition}). 


\subsection{Random SDPs} \label{subsection:random-SDP}
Our final experiment is to show the scalability of the inner approximations in Algorithm \ref{Algorithm:inner approximation}. We generated seven random large-scale SDPs with PSD constraints of $1500$, $2000$, $2500$, $3000$, $3500$, $4000$, and $4500$. The number of linear constraints is fixed as $m = 10$. We approximate the PSD cone using two different partitions
$\alpha = \{10,\ldots,10\}$ and $\beta = \{20,\ldots,20\}$. As discussed in Remark \ref{remark:partition}, we failed to use SDD approximation in this large-scale experiment. 

We ran Algorithm \ref{Algorithm:inner approximation} for 30 minutes and then compare the solution quality. The optimality gap is computed by $\lvert \frac{p^{\star}-f_{30}}{p^{\star}} \rvert \times 100 \% $, where $p^\star$ is the optimal cost value of original SDP, and $f_{30}$ is the obtained upper bound after running $30$ minutes. The results are listed in Table \ref{Table:randomSDP}.  Our proposed method shows promising efficiency and accuracy. For example, when $n=4500$ and $m=10$, Algorithm \ref{Algorithm:inner approximation} with partition $\beta$ obtained a solution with $99.9\%$ optimality in $30$ minutes, while original SDP took over 4.5 hours to solve. 


\begin{table}[t]
  \begin{center}
  \setlength{\abovecaptionskip}{1pt}
    \caption{Computational results of 6 different large-scale SDPs using Algorithm \ref{Algorithm:inner approximation} with $\alpha=\{10,\ldots,10\}$ and $\beta=\{20,\ldots,20\}$. $f_1$ denotes the cost value of the first iteration. $f_{30}$ denotes the cost value after $30$ minutes. The time consumption  (in seconds) for solving the original SDP is listed in the last column. }
    \small
    \begin{tabular}{p{6mm}p{7mm}p{7mm}p{6mm}p{7mm}p{7mm}p{6mm}p{6.5mm}}
    \toprule
    & \multicolumn{3}{c}{$\mathcal{FW}^n_{\alpha,2}$} & \multicolumn{3}{c}{$\mathcal{FW}^n_{\beta,2}$} & \multicolumn{1}{c}{PSD} \\
    \cmidrule(lr){2-4} \cmidrule(lr){5-7} \cmidrule(lr){8-8}
    $n$&$f_1$&$f_{30}$&Gap&$f_1$&$f_{30}$&Gap&  Time\\
    \toprule
    $1500$  & $5.63e6$& $4.76e6$ & $0.03$& $5.20e6$& $4.76e6$ & $0.03$& $603$ \\
    $2000$  & $3.33e6$ & $2.86e6$ & $0.10$ & $3.09e6$ & $2.86e6$ & $0.05$&  $1\,201$  \\
    $2500$  & $6.11e6$ & $5.29e6$ & $0.07$ &$5.70e6$ & $5.29e6$ & $0.05$& $2\,893$  \\
    $3000$  & $1.81e7$ &$1.32e7$ & $0.79$ &$1.57e7$ &$1.32e7$ & $0.79$&  $5\,508$  \\
    $3500$  & $8.96e6$&$7.08e6$ & $0.10$ &$8.02e6$&$7.07e6$ & $0.08$& $7\,369$\\
    $4000$  & $9.52e6$&$6.89e6$ & $0.15$ &$8.21e6$&$6.89e6$ & $0.11$& $10\,689$\\
    $4500$  & $2.05e7$& $1.70e7$ & $0.08$ &$1.88e7$& $1.69e7$ & $0.06$&  $16\,989$ \\
    \toprule
    \end{tabular}
    \label{Table:randomSDP}
  \end{center}
\end{table}

\balance 
\section{Conclusions} \label{section:Conclution}
In this paper, we have introduced the iterative inner/outer approximations for solving SDPs (cf. Algorithm \ref{Algorithm:inner approximation}-\ref{Algorithm:outer approximation}), and analyzed their solution quality (cf. Propositions \ref{prop:decreasing}-\ref{prop:increasing} and Theorems \ref{theorem:strictly decreasing}-\ref{theorem:strictly increasing}). 
Numerical results on stable set, binary quadratic optimization, and random SDPs have shown promising accuracy and computational scalability when proper partitions were used.  
Future work includes analyzing the convergence of (or modified) Algorithm \ref{Algorithm:inner approximation}-\ref{Algorithm:outer approximation} (some recent results appeared in \cite{roig2022globally}). Developing other types of iterative algorithms based on block factor-width matrices will also be interesting. 

\addtolength{\textheight}{-12cm}   






\bibliographystyle{IEEEtran}
\bibliography{references}

\end{document}